\documentclass{amsart} 
\usepackage{amsfonts}
\usepackage{amsmath}
\usepackage{amsthm}
\usepackage{amssymb}
\usepackage{setspace}
\usepackage[all]{xy}
\usepackage{color}
\usepackage{colortbl}
\usepackage{tikz}
\usepackage{calc, xkeyval, ifthen, todonotes}
\usetikzlibrary{arrows.meta}
\usetikzlibrary{patterns}
\usepackage[margin=1.25in]{geometry}
\usepackage[hidelinks]{hyperref}
\usepackage{comment}
\usepackage{soul}

\definecolor{job}{RGB}{200,65,0}
\definecolor{eric}{RGB}{155,155,255}

\newtheorem{theorem}{Theorem}[section] 
\newtheorem{proposition}[theorem]{Proposition} 
\newtheorem{lemma}[theorem]{Lemma} 
\newtheorem{corollary}[theorem]{Corollary} 

\theoremstyle{definition}
\newtheorem{definition}[theorem]{Definition} 
\newtheorem{remark}[theorem]{Remark} 
\newtheorem{example}[theorem]{Example} 
\newtheorem{notation}[theorem]{Notation} 
\newtheorem{convention}[theorem]{Convention} 

\DeclareMathOperator{\Hom}{Hom}
\DeclareMathOperator{\supp}{supp}

\newcommand{\R}{{\mathbb R}}
\newcommand{\Z}{{\mathbb Z}}
\newcommand{\N}{{\mathbb N}}
\newcommand{\Sp}{{\mathbb S}}
\newcommand{\At}{{\widetilde{A}}}
\newcommand{\ARt}{{\At_\R}}

\newcommand{\kVec}{{k\text{-}Vec}}

\newcommand{\vecb}{\vec{b}}
\newcommand{\Vhat}{\widehat{V}}
\newcommand{\Vhatdown}{\widehat{V}_{\downarrow}}
\newcommand{\Vhatup}{\widehat{V}_{\uparrow}}

\newcommand{\dbspace}{}

\newcommand{\textdef}[1]{\textbf{#1}}

\title{Decomposition of Pointwise Finite-Dimensional $\Sp^1$ Persistence Modules}

\subjclass[2010]{16G20, 55N31}
\keywords{persistent homology, circular persistence, angle-valued persistence, quiver representations, continuous quivers}

\author{Eric J.\ Hanson}
\address{Department of Mathematics, Brandeis University}
\curraddr{}
\email{ehanson4@brandeis.edu}
\thanks{}

\author{Job D.\ Rock}
\address{Department of Mathematics, Brandeis University}
\curraddr{}
\email{jobrock@brandeis.edu}
\thanks{}

\date{6 July 2020}

\dedicatory{To Kiyoshi Igusa and Gordana Todorov}

\begin{document}

\begin{abstract} We prove that pointwise finite-dimensional $\Sp^1$ persistence modules over an arbitrary field decompose uniquely, up to isomorphism, into the direct sum of a bar code and finitely-many Jordan cells.
These persistence modules have also been called angle-valued or circular persistence modules.
We allow either a cyclic order or partial order on $\Sp^1$ and do not have additional finiteness requirements on the modules.
We also show that a pointwise finite-dimensional $\Sp^1$ persistence module is indecomposable if and only if it is a bar or Jordan cell (a string or a band module, respectively, in representation theory).
Along the way we classify the isomorphism classes of such indecomposable modules. \end{abstract}

\maketitle
\tableofcontents

\section{Introduction}
\subsection{History}
Persistent homology is an essential tool in topological data analysis.
Outside of pure mathematics it has been used in studying the structure of silica glass \cite{HNHEMN} and biological aggregation (swarms) \cite{TZH}.
Mathematical applications include Floer-Novikov theory \cite{UZ} and Morse theory \cite{PSS}.
Burghelea and Dey study persistence modules from certain angle-valued maps (called  circular persistence in \cite{V-J}) and show they decompose into a bar code and Jordan cells \cite{BD}.
Burghelea and Dey's decomposition has been an important component in several other works \cite{DMW,B,BH,CM,D}.
There have been recent constructions similar to $\Sp^1$ persistence in the work by Igusa and Todorov, by Guillermou, and by Sala and Schiffmann \cite{IT13,G,SS}.

Explicitly used in Burghelea and Dey's work is the decomposition of representations of type $\At_n$ quivers. This fits into a recent trend of studying the decomposition of pointwise finite-dimensional (pwf) persistence modules from the perspective of representation theory. A survey of the connection between the decomposition of (discrete) linear persistence modules and represntations of type $A$ quivers can be found in \cite{O}. In the continuous setting, Crawley-Boevey proved pwf persistence modules over $\R$ decompose into a direct sum of bars, or a bar code \cite{C-B}.
Botnan and Crawley-Boevey further proved that pwf persistence modules of the continuous zigzag (or continuous type $A$) decompose in a similar way \cite{BC-B}.
The second author, with Igusa and Torodov, recovered this result via a representation-theoretic approach \cite{IRT}.

\subsection{Contributions and Outline}
The purpose of this paper is to generalize explicit decompositions such as those in \cite{BD, G, SS} to all pwf $\Sp^1$ persistence modules (Theorem \ref{thm:decomposition}).
We do not require further finiteness on modules and we allow either the standard cyclic order or a partial order on $\Sp^1$ (see Remark \ref{rmk:BD finitistic}).
We also provide a description of indecomposable pwf $\Sp^1$ persistence modules (Corollary \ref{cor:complete indecomposable classification}).

We introduce continuous quivers of type $\At$ as categories (Definition \ref{def:ARt}) whose representations (Definition \ref{def:representation}) are exactly $\Sp^1$ persistence modules.
We then define bar(code) persistence modules, which are constructed using an interval subset of $\R$ (Definition \ref{def:string indecomposable}).
Next we define Jordan cell persistence modules which are determined by a linear map ``traveling around'' $\Sp^1$ (Definition \ref{def:band indecomposable}); for a module $V$ we denote this map by $\widehat{V}$.
Before proceeding we show that the isomorphism classes of bars and Jordan cells are analagous to the those for type $\At_n$ quivers:
\begin{theorem}[Theorem \ref{thm:isomorphism classes}]
Let $V$ and $W$ be pwf $\Sp^1$ persistence modules.
\begin{enumerate}
    \item If $V$ and $W$ are bars then $V$ is isomorphic to $W$ if and only if they defined by the same interval of $\R$ modulo $2\pi$.
    \item If $V$ and $W$ are Jordan cells then $V$ is isomorphic to $W$ if and only if there exists a conjugation of transformations $\widehat{V} = A^{-1}\widehat{W} A$.
    \item If $V$ is a bar and $W$ is a Jordan cell then $V$ is not isomorphic to $W$.
\end{enumerate}
\end{theorem}

We define finitistic $\Sp^1$ persistence modules (Definition \ref{def:finitistic}), each of which determines a partition $\mathcal P_V$ on $\Sp^1$ (Definition \ref{def:the partition}) that yields an auxiliary quiver $Q_V$ of type $\At_{|\mathcal P_V|-1}$ (Definition \ref{def:auxiliary quiver}).
We show that direct sums of indecomposable representations of $Q_V$ can be ``pushed down'' to indecomposable $\Sp^1$ persistence modules (Lemmas \ref{lem:push down}, \ref{lem:indecomposable push down}, and \ref{lem:decomposition push down}).
We use this technique to prove the main results about our decomposition, summarized below.
\begin{theorem}[Theorems \ref{thm:indecomposables} and \ref{thm:decomposition}, Corollary \ref{cor:complete indecomposable classification}]
Let $V$ be a pointwise finite-dimensional $\Sp^1$ persistence module.
\begin{enumerate}
    \item $V$ is indecomposable if and only if it is either a bar or Jordan cell.
    \item $V$ decomposes uniquely (up to isomorphism) into a direct sum of a bar code and finitely-many Jordan cells.
\end{enumerate}
\end{theorem}

Our contributions are the extension of known decomposition theorems, the explicitness of an arbitrary decomposition, and the complete description of indecompsoable pwf $\Sp^1$ persistence modules.
Botnan and Crawley-Boevey prove that a pwf persistence module over a small category must have a decomposition into indecomposable modules, each with local endomorphism ring, that is unique up to isomorphism \cite{BC-B}.
However, their result does not provide a description of the indecomposable summands because in full generality this is not possible.
\color{black}Botnan and Crawley-Boevey's complete result combined with ours yields the following corollary, generalizing a well-known property about indecomposable type $\At_n$ representations.

\begin{corollary}[Corollary \ref{cor:local endomorphism ring}]
\color{black}A pwf $\Sp^1$ persistence module $V$ has a local endomorphism ring if and only if it is either a bar or a Jordan cell.
\end{corollary}

\subsection{Further Study}
The decomposition of pwf representations of $\At_n$ into bars and Jordan cells (or string and band modules) extends to the a larger class called string algebras (see \cite{BR,C-B18}). We suspect our decomposition of $\Sp^1$ persistence modules likewise extends to some larger class of ``continuous string algebras.''

We expect the \emph{category} of finitely-generated $\Sp^1$ persistence modules to be similar to that of $\At_n$ pwf representations. Furthermore, similar objects to pwf $\Sp^1$ persistence modules were used by Igusa and Todorov \cite{IT13} to construct Continuous Frobenius categories which were later used in studying (continuous) cluster categories \cite{HJ, IT15}; this connection requires further investigation.

Finally, continuous analogues for the decomposition theorems of persistence modules of other types of quivers (for example, the commutative ladder quivers used by Escolar and Hiraoka \cite{EH}) remain an open question.

\section{$\Sp^1$ Persistence Modules}
Let $k$ be a field.
We will denote the elements of $\Sp^1$ by $e^{i\theta}$ corresponding to the unit circle in $\mathbb C$.
By $\xi$ we denote the standard covering $\xi:\R\to\Sp^1$ given by $\theta\mapsto e^{i\theta}$.

In this section we define the types of partial orders we allow on $\Sp^1$ and its persistence modules in this context.
One may think of the partial order on $\Sp^1$ as obtained from some certain partial order on $\R$ via the 1 point compactification.
Such partial orders are described by the second author along with Igusa and Todorov in \cite{IRT} as having finitely-many sinks and sources. 
For example, one may invert $\R$'s standard order on $(-\infty,0]$ and then compactifiy $\R$. 
This partial order yields Example \ref{xmp:first example} (1), shown in Figure \ref{fig:first example}.
We construct a category from this intuition, drawing inspiration from \cite{IRT}, starting with Definitions \ref{def:the S's} and \ref{def:the R's}, which will be useful throughout the rest of the paper.
\begin{definition}\label{def:the S's}
Let $S\subset\Sp^1$ such that $|S|$ is even (possibly 0).
If $S\neq \emptyset$, elements of $S$ are indexed by an interval subset of $\N$ containing 0 such that $s_n$ denotes the element of $S$ with index $n$.
The elements of $S$ are indexed counterclockwise starting from a chosen $s_0$ such that there exists $\beta,\gamma$ where $0<\gamma-\beta<2\pi$, $s_0=e^{i\beta}$, and $s_{|S|-1}=e^{i\gamma}$.
Additionally, for all $s_n\in S$ there exists $\theta$ such that $s_n=e^{i\theta}$ and $\beta \leq \theta \leq \gamma$.
For convenience we denote by $s_{|S|}$ the element $s_0\in S$.
\end{definition}

\begin{definition}\label{def:the R's}
Let $S\subset \Sp^1$ be as in Definition \ref{def:the S's}.
\begin{itemize}
\item If $S=\emptyset$ define $\alpha_0=0$, $\alpha_1=\pi$, and $\alpha_2=2\pi$. 
\item If $S$ is nonempty define $\alpha_0\in [0,2\pi)$ such that $e^{i\alpha_0}=s_0$. 
For $0<n<|S|$, let $\alpha_n \in (\alpha_0, \alpha_0+2\pi)$ such that $s_n = e^{i\alpha_n}$.
Let $\alpha_{|S|} = \alpha_0+2\pi$.
\end{itemize}
For each $\alpha_0\leq \alpha_n < \alpha_0+2\pi$ define
\begin{align*} R_n &= \{e^{i\theta} : \alpha_n < \theta <\alpha_{n+1}\} & \overline{R}_n &= \{e^{i\theta} : \alpha_n \leq \theta \leq \alpha_{n+1}\}. \end{align*}
Furthermore, define $\mathcal R = \{R_n\}$ and note $\bigcup_{\mathcal R} \overline{R}_n = \Sp^1$.
\end{definition}

\begin{definition}\label{def:ARt}
A \textdef{continuous quiver of type $\At$}, denoted $\ARt$, is a category from $\Sp^1$ and $S$ constructed in the following way.
Let $S$ and $\mathcal R$ be as in Definitions \ref{def:the S's} and \ref{def:the R's}.
The objects of $\ARt$ are the elements of $\Sp^1$.
There is a \textdef{generating morphism} $g_{x,y}:x\to y$ if and only if there exists $\theta,\phi\in\R$ where $x=e^{i\theta}, y=e^{i\phi}$ such that the following hold.
\begin{itemize}
\item If $S=\emptyset$ then the inequality $0\leq \phi - \theta < 2\pi$ is satisfied.
\item If $S\neq\emptyset$, then\begin{enumerate}
    \item There exists $R_n\in\mathcal R$ such that $\theta,\phi\in[\alpha_n,\alpha_{n+1}]$.
    \item If $n$ is even then $\phi\leq \theta$ and if $n$ is odd then $\theta\leq\phi$.
When $|S|=2$ there are \emph{two} generating morphisms $g^{\uparrow}_{s_1,s_0},g^{\downarrow}_{s_1,s_0}:s_1\to s_0$.
\end{enumerate}
\end{itemize}
When $\theta=\phi$ we define the generating morphism to be the identity.

Let $g_{e^{i\theta}, e^{i\phi}}$ and $g_{e^{i\phi}, e^{i\psi}}$ be generating morphisms. To define $g_{e^{i\phi}, e^{i\psi}}\circ g_{e^{i\theta}, e^{i\phi}}$, suppose without loss of generality $|\phi-\theta| <2\pi$ and $|\psi-\phi|<2\pi$.
When $|S|=2$, $e^{i\theta}=s_1$, and $s^{i\psi}=s_0$ the composition is $g^{\uparrow}_{s_1,s_0}$ if $\theta < \psi$ and $g^{\downarrow}_{s_1,s_0}$ if $\psi < \theta$.
In other cases where $|\psi-\theta|<2\pi$ we define the composition to be the generating morphism $g_{e^{i\theta}, e^{i\psi}}$.
If none of these are satisfied, $g_{e^{i\phi}, e^{i\psi}}\circ g_{e^{i\theta}, e^{i\phi}}$ is a distinct morphism.
\end{definition}

\begin{notation}\label{note:order}
Let $\ARt$ be a continuous quiver of type $\At$.
If there is a generating morphism $x\to y$ for $x,y\in \Sp^1$ we write $y\preceq x$.
\end{notation}

\begin{remark}[Ordering]\label{rmk:order}
	If $S=\emptyset$ we have the standard counterclockwise cyclic order on $\Sp^1$
	(By symmetry this covers the clockwise cyclic order as well).
	
	If $S\neq\emptyset$ then $\preceq$ is a partial order that satisfies the following properties:
	\begin{enumerate}
		\item The order ``reverses'' at each element of $S$.
		\item If $s_n \in S$ with $n$ even, then $s_n$ is a sink. That is, if $x \preceq s_n$ then $x = s_n$.
		\item If $s_n \in S$ with $n$ odd, then $s_n$ is a source. That is, if $s_n \preceq x$ then $x = s_n$.
	\end{enumerate}
\end{remark}

\begin{remark}[Hom sets]\label{rmk:hom sets}
Suppose $S=\emptyset$.
For each $x\in\Sp^1$ there is a morphism $\omega_x:x\to x$ obtained as the composition $g_{y,x}\circ g_{x,y}$ for any other $y\neq x$ in $\Sp^1$.
Thus, for $x\neq y$ in $\Sp^1$ we have
\begin{displaymath} \Hom_{\ARt}(x,y)=\{g_{x,y}\circ (\omega_x)^n: n\geq 0\} =\{ (\omega_y)^n \circ g_{x,y}: n\geq 0\}, \end{displaymath} where $(\omega_x)^0=1_x$ and $(\omega_y)^0=1_y$.

Now suppose $|S|>2$. Then we have
\begin{displaymath}
\Hom_{\ARt}(x,y) = \begin{cases} \{g_{x,y}\} & y\preceq x \\ \emptyset & \text{otherwise}. \end{cases}
\end{displaymath}
In the special case when $|S|=2$ (see Example \ref{xmp:first example} (1)) we have $\Hom_{\ARt}(s_1,s_0)=\{g^{\uparrow}_{s_1,s_0},g^{\downarrow}_{s_1,s_0}\}$ but all other $\Hom$ sets are as described above for $S\neq\emptyset$.

Note that this shows that a continuous quiver of type $\At$ is a \emph{small} category since the objects also form a set.
\end{remark}

\begin{example}\label{xmp:first example}\label{xmp:good R}
We give two examples of a continuous quiver of type $\At$. Visualizations of the partial orders are depicted in Figure \ref{fig:first example}.
\begin{enumerate}
\item
We now formally observe the example from the beginning of this section.

Let $S\subset \Sp^1$ be $\{e^{i3\pi/2}, e^{5i\pi/2}\}$ where $s_0=e^{3i\pi/2}$ and $s_{1}=e^{i5\pi/2}$. 
If $\theta<\phi$ in $[\frac{3\pi}{2},\frac{5\pi}{2}]$ then $e^{i\theta}\preceq e^{i\phi}$.
If $\theta<\phi$ in $[\frac{5\pi}{2} , \frac{7\pi}{2}]$ then $e^{i\phi}\preceq e^{i\theta}$.
\smallskip

\item 
We will use this example for Examples \ref{xmp:string} and \ref{xmp:band}.

Let $S=\{e^0, e^{i\pi/2}, e^{i\pi}, e^{i3\pi/2}\}$ where $s_0=e^0$, $s_1=e^{i\pi/2}$, $s_2=e^{i\pi}$, and $s_3=e^{i3\pi/2}$.
If $\theta < \phi$ in $[0,\frac{\pi}{2}]$ or $[\pi,\frac{3\pi}{2}]$ then $e^{i\theta}\preceq e^{i\phi}$.
If $\theta < \phi$ in $[\frac{\pi}{2},\pi]$ or $[\frac{3\pi}{2},2\pi]$ then $e^{i\phi}\preceq e^{i\theta}$.
\end{enumerate}
\end{example}

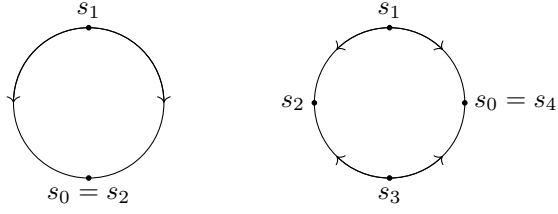
\begin{figure}	\begin{center}	\begin{tikzpicture}
\draw[white] (-3,0) -- (7,0);

\draw (1,0) arc(0:360:1);
\draw[->] (0,1) arc(90:180:1);
\draw[->] (0,1) arc(90:0:1);
\filldraw (0,1) circle[radius=.3mm];
\filldraw (0,-1) circle[radius=.3mm];
\draw (0,-1) node[anchor=north] {$s_0=s_2$};
\draw (0,1) node[anchor=south] {$s_1$};

\draw (5,0) arc (0:360:1);
\filldraw (5,0) circle[radius=.3mm];
\filldraw (3,0) circle[radius=.3mm];
\filldraw (4,-1) circle[radius=.3mm];
\filldraw (4,1) circle[radius=.3mm];
\draw[->] (4,1) arc (90:135:1);
\draw[->] (4,1) arc (90:45:1);
\draw[->] (4,-1) arc (270: 225:1);
\draw[->] (4,-1) arc (270:315:1);
\draw (5,0) node[anchor=west] {$s_0=s_4$};
\draw (4,1) node[anchor=south] {$s_1$};
\draw (3,0) node[anchor=east] {$s_2$};
\draw (4,-1) node[anchor=north] {$s_3$};

\end{tikzpicture}
\caption{Example \ref{xmp:first example} (1) and (2) on the left and right, respectively.}\label{fig:first example}
\end{center}	\end{figure}

Throughout the paper we will often replace the phrase ``$\Sp^1$ persistence module'' with ``representation'' both for brevity and to indicate the point of view of our proof techniques.
Let $\kVec$ be the category of $k$-vector spaces (including infinite-dimensional spaces).

\begin{definition}\label{def:representation}
Let $\ARt$ be a continuous quiver of type $\At$.
A \textdef{representation} $V$ of $\ARt$ over $k$ is a functor $V:\ARt\to \kVec$.
If $y\preceq x\in \Sp^1$ we denote by $V(x,y)$ the morphism in $\kVec$ obtained by applying $V$ to the generating morphism $g_{x,y}$ in $\ARt$ with two caveats.

(1) If $|S|=2$ we write $V(s_1,s_0)$ to mean $V(g^{\downarrow}_{s_1,s_0})$ and write $V(s_1,s_2)$ to mean $V(g^{\uparrow}_{s_1,s_0})$.
(2) If $S=\emptyset$ we write $V(x,x)$ for the identity on $V(x)$.
The morphism obtained by applying $V$ to $\omega_x^n$, for $n\geq 1$, is written $V(\omega_x^n)$. 
\end{definition}

The reader may note that $t^n$ in the definition of an $\Sp^1$ representation by Igusa and Todorov \cite[Definition 1.1.1]{IT13} plays a similar role to $V(\omega_x^n)$ in Definition \ref{def:representation}.

Throughout the paper we will often state definitions only over generating morphisms as they indeed generate all the morphisms in $\At$.

\begin{definition}\label{def:pwf}
Let $\ARt$ be a continuous quiver of type $\At$.
We call a representation $V$ of $\ARt$ \textdef{pointwise finite-dimensional} (or \textdef{pwf}) if $\dim V(x)<\infty$ for all $x\in \Sp^1$.
\end{definition}

\begin{definition}\label{def:morphism}\label{def:isomorphism}
Let $\ARt$ be a continuous quiver of type $\At$ and $V, W$ representations of $\ARt$ over $k$.
A \textdef{morphism of representations} $f:V\to W$ is a natural transformation from $V$ to $W$.
I.e., the following diagram commutes for each generating morphism $g_{x,y}$ in $\ARt$.
\begin{displaymath}\xymatrix{
V(x) \ar[r]^-{V(x,y)} \ar[d]_-{f(x)} & V(y) \ar[d]^-{f(y)} \\ W(x) \ar[r]_-{W(x,y)} & W(y).
}\end{displaymath}

If $f(x)$ is an isomorphism for all $x\in \Sp^1$ we call $f$ an \textdef{isomorphism}.
If there is an isomorphism $f:V\to W$ we say $V$ and $W$ are \textdef{isomorphic} and write $V\cong W$.
\end{definition}

\begin{remark}\label{rmk:isomorphism}
Let $f:V\to W$ be a morphism of representations.
If $f$ is an isomorphism then by letting $g(x)= f^{-1}(x)$ for all $x\in\Sp^1$ we obtain a morphism $g:W\to V$.
Furthermore, $(g\circ f)(x)$ is the identity on each $V(x)$ and $(f\circ g)(x)$ is the identity on each $W(x)$, thus justifying the name isomorphism.
\end{remark}

\begin{definition}\label{def:direct sum}
The \textdef{direct sum} $\bigoplus_{\zeta\in Z} V_\zeta$ of a multi-set of representations $\{V_\zeta\}_{\zeta\in Z}$ has vector spaces $\left(\bigoplus_{\zeta\in Z} V_\zeta\right)(x) = \bigoplus_{\zeta\in Z} \left(V_\zeta(x)\right)$ and linear maps given by the direct sum of linear maps on each of the vector space summands. Each $V_\zeta$ is a \textdef{summand} of $\bigoplus_{\zeta\in Z} V_\zeta$. We say the direct sums $\bigoplus_{\zeta\in Z} V_\zeta$ and $\bigoplus_{\chi\in X} V_\chi$ are the same \textdef{up to isomorphism} if there exists a bijection $\Phi:Z \rightarrow X$ such that $V_\zeta \cong V_{\Phi(\zeta)}$ for all $\zeta \in Z$.
\end{definition}

\begin{remark}\label{rmk:direct sum}
These two properties follow immediately from Definition \ref{def:direct sum}.
\begin{enumerate}
    \item Any two direct sums which are the same up to isomorphism are isomorphic, but the converse is not true in general.
    \item A representation $U$ is a summand of a representation $V$ if and only if there exist morphisms $f:U\to V$ and $g:V\to U$ such that $g\circ f= 1_U$.
\end{enumerate}
\end{remark}

\begin{definition}
Let $\ARt$ be a continuous quiver of type $\ARt$ and $V$ a representation of $\ARt$.
Let $R\subset \Sp^1$ be an arbitrary subset.
Recalling our notation in Definition \ref{def:representation}, define the \textdef{restriction of $V$ to $R$}, denoted $V|_R$, to be
\begin{align*}
V|_R(x) &= \begin{cases} V(x) & x\in R \\ 0 & \text{otherwise} \end{cases} \\
V|_R(x,y) &= \begin{cases} V(x,y) & V(z,y)\circ V(x,z)=V(x,y) \Rightarrow z\in R \\ 0 & \text{otherwise} \end{cases}
\end{align*}
Note if $S=\emptyset$ then $V(\omega_x)=0$ for all $x\in\Sp^1$ unless $R=\Sp^1$.
\end{definition}

The following lemma is essential to prove our results in Section \ref{sec:the thms}.
\begin{lemma}\label{lem:good split}
Let $\ARt$ be a continuous quiver of type $\At$ and $V$ a representation of $\ARt$.
Let $\beta,\gamma\in\R$ such that $0<\gamma-\beta\leq 2\pi$ and let
\begin{align*}
R&=\{e^{i\theta} :\beta<\theta<\gamma\} &
\overline{R}&= \{e^{i\theta} : \beta\leq \theta\leq \gamma\}.
\end{align*}
Suppose there exists a representation $U$ of $\ARt$ such that $\supp U\subset R$.
If $U$ is a summand of $V|_{\overline{R}}$ then $U$ is a summand of $V$.
\end{lemma}
\begin{proof}
By assumption there exists morphisms $f:U\to V|_{\overline{R}}$ and $g:V|_{\overline{R}}\to U$ such that $g\circ f = 1_U$.
For all $x\in \Sp^1$, define $\tilde{f}(x):U(x)\to V(x)$ and $\tilde{g}(x):V(x)\to U(x)$:
\begin{align*}
\tilde{f}(x) &= \begin{cases} f(x) & x\in R \\ 0 & \text{otherwise}. \end{cases} &
\tilde{g}(x) &= \begin{cases} g(x) & x\in R \\ 0 & \text{otherwise}. \end{cases}
\end{align*}
By definition $\tilde{g}(x) \circ \tilde{f}(x) = 1_{U(x)}$ for all $x\in\Sp^1$.
We must show that $\{\tilde{f}(x)\}$ and $\{\tilde{g}(x)\}$ yield morphisms $\tilde{f}:U\to V$ and $\tilde{g}:V\to U$.

Let $g_{x,y}$ be a generating morphism in $\ARt$ and consider the diagram
\begin{displaymath}\xymatrix{
U(x) \ar[d]_-{U(x,y)} \ar[r]^-{\tilde{f}(x)} & V(x) \ar[d]^-{V(x,y)} \ar[r]^-{\tilde{g}(x)} & U(x) \ar[d]^-{U(x,y)} \\
U(y) \ar[r]^-{\tilde{f}(y)} & V(y) \ar[r]^-{\tilde{g}(y)} & U(y) .
}\end{displaymath}
If there exists $z\in\overline{R}\setminus R$ such that $g_{x,y}=g_{z,y}\circ g_{x,z}$ then $V(z,y)\circ (x,z)\circ \tilde{f}(x)=0$ by assumption.
In this case $U(x,y)=0$ as well and so the diagram commutes.
It is clear the diagram commutes if there is no such $z\in\overline{R}\setminus R$ or when $x,y\notin R$.
\end{proof}

\begin{remark}[Relation to known work]\label{rmk:BD finitistic}\label{rmk:other S^1 persistence}
Burghelea and Dey's original definition \cite{BD} yields finitistic $\Sp^1$ persistence modules where $\Sp^1$ has a partial order.
In Sala and Schiffmann's work \cite[\S 5.2]{SS}, coherent persistence modules are certain pwf $\Sp^1$ persistence modules where $\Sp^1$ has cyclic order.
A pwf $\Sp^1$ persistence module $V$ in either of \cite{BD,SS} partitions $\Sp^1$ into finitely-many pieces and on each piece $V$ must be essentially constant.
Constructible sheaves on the circle in Guillermou's work \cite[\S 4.3]{G} relaxes some conditions for cyclic $\Sp^1$ but does not include partial orders on $\Sp^1$.
Both \cite{BD,G} decompose their respective modules into bars and Jordan cells.
\end{remark}

\section{Bar Codes and Jordan Cells}
In this section we describe the bar (string) and Jordan cell (band) representations.
We will show both representations are indecomposable in Section \ref{sec:the thms}.

\begin{definition}\label{def:I and E}
Let $I\subset \R$ be a bounded interval inheriting the ordering from $\R$.
For each $e^{i\theta} \in\Sp^1$ define $E(e^{i\theta}) = \xi^{-1}(e^{i\theta}) \cap I$ also inheriting the ordering from $\R$.
\end{definition}

\begin{definition}\label{def:string indecomposable}
Let $\ARt$ be a continuous quiver of type $\At$.
Let $I$ and the $E(e^{i\theta})$'s be as in Definition \ref{def:I and E}.
We define a representation $M_I$ with vector spaces $M_I(e^{i\theta})$:
\begin{displaymath}
M_I(e^{i\theta}) = k^{|E(e^{i\theta})|}.
\end{displaymath}

For each $e^{i\theta}\in\Sp^1$ we do the following.
Let $\mathcal B(e^{i\theta})$ be the standard ordered basis of $M_I(e^{i\theta})$.
There is a unique order-preserving bijection $E(e^{i\theta}) \to \mathcal B(e^{i\theta})$.
Denote by $\vecb_\beta$ the image of $\beta$ in $\mathcal B(e^{i\phi})$.

Let $\theta,\phi\in[\alpha_n,\alpha_{n+1}]$ such that $g_{e^{i\theta}, e^{i\phi}}$ is a generating morphism in $\ARt$.
We define $M_I(e^{i\theta},e^{i\phi})$ on each basis vector $\vecb_\beta\in M_I(e^{i\theta})$ by
\begin{displaymath}
M_I(e^{i\theta},e^{i\phi})(\vecb_\beta) = \begin{cases}
\vecb_{\beta-\theta+\phi} &  \beta -  \theta + \phi \in E(e^{i\phi}) \\
0 & \text{otherwise}. \end{cases}
\end{displaymath}
In the case $S=\emptyset$, $e^{i\theta}\in R_0$, and $e^{i\phi}\in R_1$:  $M_I(e^{i\theta},e^{i\phi})$ and $M_I(e^{i\phi},e^{i\theta})$ are the compositions $M_I(e^{i\pi},e^{i\phi})\circ M_I(e^{i\theta},e^{i\pi})$  and $M_I(e^0,e^{i\theta})\circ M_I(e^{i\phi},e^0)$, respectively.
A representation isomorphic to $M_I$ is a \textdef{bar}.
A direct sum of bars is a \textdef{bar code}.
\end{definition}

In Definition \ref{def:string indecomposable} the idea behind the linear maps is to send a basis vector in $M_I(e^{i\theta})$ to the ``next'' basis vector in $M_I(e^{i\phi})$.
Replacing $\beta-\theta+\phi$ with $\gamma$ we have the equation $\beta-\gamma = \theta-\phi$ if $\gamma\in E(e^{i\phi})$; this is precisely what we want.

\begin{example}\label{xmp:string}
	We consider the orientation of $\At_\R$ as in Example \ref{xmp:good R} (2).
	\begin{enumerate}
		\item For the interval $\left(0, \frac{3\pi}{2}\right)$ and $\theta\neq \phi\in[0,2\pi]$ where $e^{i\phi}\preceq e^{i\theta}$ we have
			\begin{align*} M_{\left(0, \frac{3\pi}{2}\right)}(e^{i\theta}) &= \begin{cases}k & 0 < \theta < \frac{3\pi}{2} \\ 0 & \text{otherwise} \end{cases} &
			M_{ \left(0, \frac{3\pi}{2}\right) }(e^{i\theta},e^{i\phi}) &= \begin{cases} 1_k & 0 < \theta \leq \phi < \frac{3\pi}{2} \\ 0 & \text{otherwise}. \end{cases} \end{align*} 
			
		\item For the interval $\left[-\pi,\frac{7\pi}{2}\right]$ and $\theta\neq\phi\in (0,2\pi]$ where $e^{i\phi}\preceq e^{i\theta}$ we have 
		  
			\begin{align*} M_{\left[-\pi,\frac{7\pi}{2}\right]}(e^{i\theta}) &= \begin{cases} k^3 & \pi \leq \theta \leq \frac{3\pi}{2} \\ k^2 & \text{otherwise} \end{cases} &
			M_{\left[-\pi,\frac{7\pi}{2}\right]}(e^{i\theta},e^{i\phi}) &= \begin{cases} 1_{k^3} & \pi \leq \theta \leq \phi \leq \frac{3\pi}{2} \\
			    A & \frac{3\pi}{2} = \theta < \phi \leq 2\pi\\ B & \frac{\pi}{2} \leq \theta < \phi = \pi \\ 1_{k^2} & \text{otherwise}. \end{cases} \end{align*}
			   with
			     \begin{align*} A &= {\begin{pmatrix} 1 & 0 & 0 \\ 0 & 1 & 0 \end{pmatrix}} &
		    B &= {\begin{pmatrix} 0 & 0 \\ 1 & 0 \\ 0 & 1 \end{pmatrix}}. \end{align*}
			    A visual depiction is in Figure \ref{fig:string}.

			\begin{figure}	\begin{center}
			\begin{tikzpicture}
			\draw (1,0) arc (0:360:1);
			\filldraw (1,0) circle[radius=.3mm];
			\filldraw (-1,0) circle[radius=.3mm];
			\filldraw (0,-1) circle[radius=.3mm];
			\filldraw (0,1) circle[radius=.3mm];
			\draw[->] (0,1) arc (90:135:1);
			\draw[->] (0,1) arc (90:45:1);
			\draw[->] (0,-1) arc (270: 225:1);
			\draw[->] (0,-1) arc (270:315:1);
			\draw[color=red,thick, domain=-3.14:11,variable=\t,smooth,samples=75] plot ({\t r}: 1.157 + .05*\t);
			\filldraw [color=red] (-1,0) circle[radius=.3mm];
			\filldraw [color=red] (0,-1.707) circle[radius=.3mm];
			
			\filldraw[fill=white!85!red, draw=white!85!red] (4,0) circle[radius=1.6];
			\filldraw[fill=white, draw=white] (4,0) circle[radius=1.2];
			\draw (5,0) arc (0:360:1);
			\filldraw (5,0) circle[radius=.3mm];
			\filldraw (3,0) circle[radius=.3mm];
			\filldraw (4,-1) circle[radius=.3mm];
			\filldraw (4,1) circle[radius=.3mm];
			\draw[->] (4,1) arc (90:135:1);
			\draw[->] (4,1) arc (90:45:1);
			\draw[->] (4,-1) arc (270: 225:1);
			\draw[->] (4,-1) arc (270:315:1);
			
			\foreach \x in {1.2, 1.4, 1.6}
			{
			    \draw[color=red, thick] (4+\x,0) arc (0:360:\x);
			}
			\draw[color=red, thick, rounded corners=5pt] (4, 1.6) arc (90:130:1.6) -- (2.927,0.899) arc (140:180:1.4);
			\draw[color=red, thick, rounded corners=5pt] (4, 1.4) arc (90:130:1.4) -- (3.080, 0.771) arc (140:180:1.2);
			
			\end{tikzpicture}\caption{\dbspace Example \ref{xmp:string} (2) on the left, a bar. Visually, we think of $\xi\left[-\pi,\frac{7\pi}{2}\right]$ as a string wrapping around $\Sp^1$ counterclockwise between 2 and 3 times.
			Example \ref{xmp:band} (2) on the right, a Jordan cell. The crossings indicate the basis vectors combining while elsewhere there is the identity. We think of this as a band.}\label{fig:string}\label{fig:band}
			\end{center}	\end{figure}
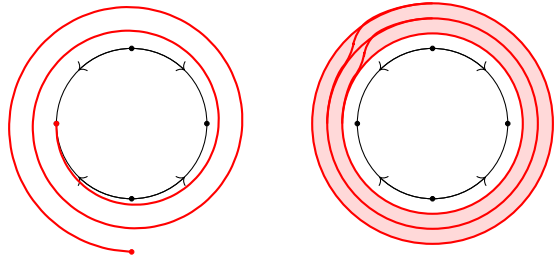
			
		    The basis elements of $M_{\left[-\pi,\frac{7\pi}{2}\right]}(e^{i\theta})$ correspond to the intersections of the string with $e^{i\theta}$ and $M_{\left[-\pi,\frac{7\pi}{2}\right]}(e^{i\theta},e^{i\phi})$ can be considered as the linear map which acts by `sliding' the basis elements down the string. 
		
		\item For the interval $\left(0,2\pi\right]$ and $\phi\neq\theta\in[0,2\pi]$ where $e^{i\phi}\preceq e^{i\theta}$ we have 
			\begin{align*} M_{\left(0,2\pi\right]}(e^{i\theta}) &= k &
			M_{\left(0,2\pi\right]}(e^{i\theta},e^{i\phi}) &= \begin{cases} 1_k & \theta \neq 0 \\ 0 & \theta = 0 \end{cases} \end{align*} 
	\end{enumerate}
\end{example}

\begin{definition}\label{def:hat map}
Let $\ARt$ be a continuous quiver of type $\At$ and let $V$ be a representation of $\ARt$.
Suppose $V(e^{i\theta},e^{i\phi})$ is an isomorphism for each generating morphism $g_{e^{i\theta},e^{i\phi}}$ in $\ARt$.  Recalling the notation in Definition \ref{def:representation}:
\begin{itemize}
    \item If $S=\emptyset$ then define $\Vhatup = V(e^{i\alpha_1},e^{i\alpha_2})$ and $\Vhatdown = V(e^{i\alpha_0},e^{i\alpha_1})^{-1}$.
    \item If $S\neq\emptyset$ then define $\Vhatup = V(e^{i\alpha_{|S|-1}},e^{i\alpha_{|S|}})$ and
    \begin{displaymath}
        \Vhatdown = V(e^{i\alpha_1},e^{i\alpha_0}) \circ V(e^{i\alpha_1},e^{i\alpha_2})^{-1}
            \circ \cdots \circ V(e^{i\alpha_{|S|-3}},e^{i\alpha_{|S|-2}})^{-1} \circ V(e^{i\alpha_{|S|-1}},e^{\alpha_{|S|-2}})
    \end{displaymath}
\end{itemize}
In both cases, define $\Vhat = \Vhatup \circ (\Vhatdown^{-1})$.
Note that if $S=\emptyset$ then $\Vhat = V(\omega_{e^0})$.
\end{definition}

\begin{definition}\label{def:band indecomposable}
Let $\ARt$ be a continuous quiver of type $\At$.
Let $V$ be a representation of $\ARt$ such that the following hold:
\begin{itemize}
\item There exists $d\geq 1$ such that $\dim V(x)=d$ for all $x\in \Sp^1$.
\item For all generating morphisms $g$, $V(g)$ is an isomorhism.
\item Let $x_0=e^0$ if $S=\emptyset$ and $x_0=s_0$ if $S\neq\emptyset$.
If $V(x_0)\cong U\oplus W$, where $U$ and $W$ are invariant subspaces under $\Vhat$, then either $U=0$ or $W=0$.
\end{itemize}
Then we call $V$ a \textdef{Jordan cell}.
\end{definition}

\begin{remark}\label{rmk:Jordan cells}
The definition of a Jordan cell is meant to be a continuous version of band representations of type $\At_n$ quivers.
The isomorphism class of a Jordan cell is completely determined by the map $\Vhat$ (see Theorem \ref{thm:isomorphism classes}). 
In particular, by a result of Kronecker, when $k$ is algebraically closed one can change the bases of $V(e^{i\alpha_0})$ and $V(e^{i\alpha_{|S|-1}})$ such that $\Vhatdown$ is the identity and $\Vhatup$ is a Jordan block.
This is why Burghelea and Dey called similar circular persistence modules `Jordan cell representations'.
\end{remark}

\begin{example}\label{xmp:band}
 We consider the orientation of $\At_\R$ as in Example \ref{xmp:good R} (2).
	\begin{enumerate}
		\item Let $\rho \in k^*$ and $\theta\neq\phi\in[0,2\pi]$ where $e^{i\phi}\preceq e^{i\theta}$. The representation $U$ defined by
		\begin{align*} U(e^{i\theta}) &= k &
		U(e^{i\theta},e^{i\phi}) &= \begin{cases}\rho & 0 \leq \phi < \theta = \frac{\pi}{2}\\1_k & \text{otherwise}\end{cases}\end{align*}
		is a 1-dimensional Jordan cell. In this case, we have $\widehat{U} = \rho^{-1}$.
		
		\item Denote by \begin{displaymath}A = {\begin{pmatrix} 1 & 1 & 0 \\ 0 & 1 & 1 \\ 0 & 0 & 1 \end{pmatrix}}.\end{displaymath}
		The representation $V$ defined by, for $\theta\neq\phi\in[0,2\pi]$ where $e^{i\phi}\preceq e^{i\theta}$, \begin{align*} V(e^{i\theta}) &= k^3 &
		V(e^{i\theta},e^{i\phi}) &= \begin{cases} A & \frac{\pi}{2} =\theta < \phi \leq \pi \\ 1_{k^3} & \text{otherwise}. \end{cases} \end{align*}
		is a 3-dimensional Jordan cell.
		In this case $\Vhat = A$. See Figure \ref{fig:band} for a visualization of this example. 
		
		\item Let $\rho,\sigma \in k^*$. Denote by
		\begin{align*} A &= {\begin{pmatrix} \rho & 0 & 0 \\ 1 & \rho & 0 \\ 0 & 1 & \rho \end{pmatrix}} &
		B &= {\begin{pmatrix} \sigma & 0 & 0 \\ 0 & 1 & 0 \\ 0 & 0 & -\sigma \end{pmatrix}}\end{align*}
		The representation $W$ defined by, for $\theta\neq\phi\in[0,2\pi]$ where $e^{i\phi},e^{i\theta}$,
		\begin{align*} 	W(e^{i\theta}) &= k^3 &
		W(e^{i\theta},e^{i\phi}) &=  \begin{cases} A & \frac{\pi}{2} \leq \theta < \phi = \pi \\ B & \pi \leq \phi < \theta = \frac{3\pi}{2} \\ 1_{k^3} & \text{otherwise}\end{cases}  \end{align*}
		is a 3-dimensional Jordan cell.
		In this case $\widehat{W} = B^{-1}A$.
	\end{enumerate}
\end{example}

The following shows that the isomorphism classes of bars and Jordan cells can be described in an analogous way to the isomorphism classes of finite-dimensional indecomposable type $\At_n$ representations.
\begin{theorem}\label{thm:isomorphism classes}
    Let $V$ and $W$ be representations of $\ARt$.
    \begin{enumerate}
        \item If $V\cong M_I$ and $W\cong M_J$ are bars then $V\cong W$ if and only if there exists an integer $n$ and bijection $I\to J$ given by $x\mapsto x+2n\pi$.
        \item If $V$ and $W$ are Jordan cells then $V\cong W$ if and only if there exists an isomorphism $A:V(e^{i\alpha_0})\to W(e^{i\alpha_0})$ such that $\widehat{V}=A^{-1} \widehat{W} A$.
        \item If $V$ is a bar and $W$ is a Jordan cell then $V\not\cong W$.
    \end{enumerate}
\end{theorem}
\begin{proof}
    (1) Suppose $V\cong M_I$ and $W\cong M_J$ are bars.
    First assume there exists $n\in\Z$ such that $x\mapsto x+2n\pi$ is a bijection $I\to J$.
    For each $x\in \Sp^1$, let $E_I(x)$ and $E_J(x)$ be as in Definition \ref{def:I and E} for $I$ and $J$, respectively.
    Then, for each $x\in\Sp^1$, we have an order preserving bijection $E_I(x)\to E_J(x)$ which induces an order preserving bijection $\widetilde{f}(x):\mathcal B_I(x) \to \mathcal B_J(x)$.
    This induces isomorphisms $f(x):V(x)\to W(x)$ such that $f(y)\circ V(x,y)=W(x,y)\circ f(x)$, yielding an isomorphism $f:V\to W$.
    
    If there is no bijection $I\to J$ as described then there exists $x\in \Sp^1$ such that $|E_I(x)|\neq |E_J(x)|$.
    Then $\dim V(x)\neq \dim W(x)$ and so $V\ncong W$.
    
    (2) Suppose $V\cong W$ are Jordan cells.
    First suppose $S\neq \emptyset$; i.e., $\ARt$ is not cyclic.
    Let $f:V\to W$ be an isomorphism and $A=f(e^{i\alpha_0}) = f(s_0)$.
    Noting that there then exists an inverse $f^{-1}:W\to V$ we see that
    \begin{align*}
       A \widehat{V} &= f(s_{|S|}) \circ V(s_{|S|-1},s_{|S|}) \circ V(s_{|S|-1},s_{|S|-2})^{-1} \circ \cdots \circ V(s_1,s_2) \circ V(s_1,s_0)^{-1} \\
       &=W(s_{|S|-1},s_{|S|})\circ f(s_{|S|-1}) \circ V(s_{|S|-1},s_{|S|-2})^{-1} \circ \cdots \circ V(s_1,s_2) \circ V(s_1,s_0)^{-1} \\
       &=W(s_{|S|-1},s_{|S|})\circ W(s_{|S|-1},s_{|S|-2})^{-1} \circ f(s_{|S|-2}) \circ \cdots \circ V(s_1,s_2) \circ V(s_1,s_0)^{-1} \\
       &\phantom{=W(s_{|S|-1},s_{|S|})\circ W(s_{|S|-1},s_{|S|-2})^{-1}\circ}\vdots \\
       &= W(s_{|S|-1},s_{|S|}) \circ W(s_{|S|-1},s_{|S|-2})^{-1} \circ \cdots \circ W(s_1,s_2) \circ f(s_1) \circ V(s_1,s_0)^{-1} \\
       &= W(s_{|S|-1},s_{|S|}) \circ W(s_{|S|-1},s_{|S|-2})^{-1} \circ \cdots \circ W(s_1,s_2) \circ V(s_1,s_0)^{-1} \circ f(s_0) \\
       &= \widehat{W} A
    \end{align*}
    Thus $\widehat{V} = A^{-1} \widehat{W} A$.
    
    Likewise, if we assume that we are given $A$ such that $\widehat{V} = A^{-1} \widehat{W} A$ we may choose each of the $f(s_i)$'s satisfying the equations above. When $\ARt$ is acyclic, consider the following diagram:
    %
    \begin{displaymath}\xymatrix@C=10ex{
    V(0) \ar[d]_-A \ar[r]^-{V(s_1,s_0)^{-1}} & V(s_1) \ar[r]^-{V(s_1,s_2)} \ar@{-->}[d]_-{f(s_1)} & \cdots \ar[r] & V(s_{|S|-2}) \ar@{-->}[d]^-{f(s_{|S|-2})} \ar[r]^-{V(s_{|S|-1},s_{|S|-2})^{-1}} & V(s_{|S|-1}) \ar@{-->}[d]^-{f(s_{|S|-1})} \ar[r]^-{V(s_{|S|-1},s_{|S|})} & V(s_{|S|}) \ar[d]^-A \\
    W(s_0) \ar[r]_-{W(s_1,s_0)^{-1}} & W(s_{1}) \ar[r]_-{W(s_1,s_2)} & \cdots \ar[r] & W(s_{|S|-2}) \ar[r]_-{W(s_{|S|-1},s_{|S|-2})^{-1}} & W(s_{|S|-1}) \ar[r]_-{W(s_{|S|-1},s_{|S|})} & W(s_{|S|}).
    }\end{displaymath}
    For the first vertical dashed arrow set $f(s_1)=W(s_1,s_0)^{-1}\circ A \circ V(s_1,s_0)$ and observe:
    \begin{align*}
    W(s_{|S|-1},s_{|S|}) \circ \cdots\circ W(s_1,s_2)\circ f(s_1) &= \underbrace{W(s_{|S|-1},s_{|S|}) \circ \cdots \circ W(s_1,s_0)^{-1}\circ A}_{\widehat{W} A} \circ V(s_1,ss_0) \\
    &= \underbrace{A \circ V(s_{|S|-1},s_{|S|}) \circ \cdots\circ V(s_1,s_0)^{-1}}_{A \widehat{V}} \circ V(s_1,s_0) \\
    &= A \circ V(s_{|S|-1},s_{|S|}) \circ \cdots \circ V(s_1s_2).
    \end{align*}
    One may now use a similar argument for $f(s_2)$ and all $x\in R_0$.
    After finitely-many iterations one obtains an isomorphism $f:V\to W$.
    
    Both arguments hold when $\ARt$ is cyclic by replacing $V(s_1,s_0)^{-1}$ and $V(s_1,s_2)$ with $V(e^{i\alpha_0},e^{i\alpha_1})$ and $V(e^{i\alpha_1},e^{i\alpha_0})$, respectively, and performing a similar replacement for $W$.
    
    (3) Since $V$ is a bar, there exists $y\preceq x$ in $\Sp^1$ such that $V(x,y)$ is not an isomorphism.
    Since $W(x,y)$ must be an isomorphism there cannot exist an isomorphism $V\to W$ or $W\to V$.
\end{proof}

\section{Refinement and Representation Lifting}
In this section we describe how to refine and lift certain kinds of representations of a continuous quiver of type $\At$.
The refinements provide an easier version to work with and the lift is a representation of a type $\At_n$ quiver.

Note, for a continuous type $\At$ quiver, each $\overline{R}_n$ and $R_n$ (from Definition \ref{def:the R's}) is totally ordered with the ordering inherited from $\ARt$.

\begin{remark}\label{rmk:restriction decomposition}
Recall Crawley-Boevey proved a pwf representation of $\R$ decomposes uniquely up to isomorphism into a bar code \cite{C-B}. Moreover, each bar is indecomposable.
For each $R_n\in \mathcal R$ consider $\overline{R}_n$ as a closed interval subset of $\R$.
Thus, since $\preceq$ on $\overline{R}$ is a total order, the restriction $V|_{\overline{R}_n}$ of a pwf representation $V$ of a continuous quiver $\ARt$ decomposes uniquely up to isomorphism into a bar code.
\end{remark}

\begin{definition}\label{def:finitistic}
Let $\ARt$ be a continuous quiver of type $\At$ and $V$ a pwf representation of $\ARt$.
We call $V$ a \textdef{finitistic} representation if for all $R_n\in \mathcal R$, the support of each bar in the bar code decomposition of $V|_{\overline{R}_n}$ has an element in $\overline{R}_n\setminus R_n$.
\end{definition}

\begin{definition}\label{def:the partition}
Let $\ARt$ be a continuous quiver of type $\At$ and $V$ a finitistic representation of $\ARt$.
Write $V|_{\overline{R}_n} = \bigoplus_{i=1}^{m_n} (A_{n,i})^{j_i}$ where each $A_{n,i}$ is a bar and $j_i$ is its multiplicity in the sum.
By Definition \ref{def:finitistic} each $A_{n,i}$ has support in $\overline{R}_n\setminus R_n$.

We define subintervals of $\overline{R}_n$ for each $A_{n,i}$.
If $A_{n,i}$ does not have support at $\alpha_{n+1}$ let $J_n^i=\supp A_{n,i}$.
If $A_{n,i}$ has support at $\alpha_{n+1}$ let $J_n^i = \overline{R}_n\setminus (\supp A_{n,i})$, which may be empty.
We let $J_n = \{J_n^i: J_n^i\neq \emptyset\}$.
Note that if $J_n^i = J_n^{i'}$ then there is only one copy in $J_n$.

If $J_n = \emptyset$ then we construct a partition $\mathcal P_n$ of $\overline{R}_n$ as follows.
We let $P_n^1$ be the closed subinterval $[\alpha_n,\frac{\alpha_n+\alpha_{n+1}}{2}]$ of $\overline{R}_n$.
Then let $P_n^2 = \overline{R}_n \setminus P_n^1$ and $\mathcal P_n = \{P_n^1,P_n^2\}$.

If $J_n\neq\emptyset$ we perform an algorithm to construct a partition $\mathcal P_n$ of $\overline{R}_n$.
Notice $J_n$ is finite and totally ordered by inclusion.
Set $\ell =0$, $J_n^0 = J_n$, and $P_n^0 = \emptyset$.
\begin{enumerate}
\item Let $P_n^{\ell+1} = (\min J_n^\ell) \setminus (\bigcup_{i=0}^\ell P_n^i)$.
\item Let $J_n^{\ell+1} = J_n^\ell \setminus \{\min J_n^\ell\}$.
\item Replace $\ell$ with $\ell+1$.
\item (i) If $J_n^\ell=\emptyset$ then set $P_n^{\ell+1} = \overline{R}_n \setminus (\bigcup_{i=0}^\ell P_n^i)$ and $\mathcal P_n= \{P_n^i : 0 < i \leq \ell +1\}$. Then exit the algorithm.
(ii) If $J_n^\ell\neq \emptyset$ then return to step 1.
\end{enumerate}

Denote by $\overline{R}_{|S|}$ and $R_{|S|}$ the sets $\overline{R}_0$ and $R_0$, respectively.
Similarly, denote by $\mathcal P_{|S|}$ and $P_{|S|}^i$, for $1\leq i \leq |\mathcal P_0|$, the sets $\mathcal P_0$ and $P_0^i$, respectively.
Our partition is the collection of $P_n^i$'s, taking the union where the ends overlap:
\begin{displaymath}
\mathcal P_V = \{P_n^i: 0\leq n < |S|,\, 1 < i < |\mathcal P_n|\} \cup \{P_n^{|\mathcal P_n|} \cup P_{n+1}^1 : 0 \leq n <|S|\}.
\end{displaymath}
Note that if $S\neq\emptyset$ then the unique $s\in P_n^{|\mathcal P_n|} \cap P_{n+1}^1$ is in $S$.
\end{definition}

\begin{proposition}\label{prop:the partition}
Let $\ARt$ be a continuous quiver of type $\At$ and $V$ a finitistic representation.
Let $\mathcal P_V$ be as in Definition \ref{def:the partition}.
Then $\mathcal P_V$ is indeed a partition of $\Sp^1$ and each $P\in\mathcal P_V$ is path-connected.
\end{proposition}
\begin{proof}
For each $x\in \Sp^1$ there is $\overline{R}_n$ such that $x\in \overline{R}_n$.
Then there is a $P_n^i$ such that $x\in P_n^i$ and thus a $P\in \mathcal P_V$ such that $x\in P$.

Consider $P_m^i,P_n^j\in \mathcal P_V$.
If $m\neq n$ or $i\neq j$ then $P_m^i\cap P_n^j=\emptyset$.
For a fixed $n$ and $P_n^i\in\mathcal P_V$ we see $P_n^i\cap P_n^1=\emptyset$ and $P_n^i\cap P_n^{|\mathcal P|}=\emptyset$.
Now consider
\begin{displaymath} \left( P_m^{|\mathcal P_m|}\cup P_{m+1}^1\right) \cap \left( P_n^{|\mathcal P_n|}\cup P_{n+1}^1 \right).
\end{displaymath}
If $m\neq n$ this intersection is empty.
Thus, if $P,P'\in \mathcal P_V$ and $P\neq P'$ then $P\cap P'=\emptyset$.

Finally, each $P_n^i$ is path-connected and $P_n^{|\mathcal P_n|}\cap P_{n+1}^0$ contains exactly one element.
Thus, each $P\in \mathcal P_V$ is path-connected.
\end{proof}

We will often reuse the hypotheses in the preceding proposition as ``the setup in Proposition \ref{prop:the partition}.''

\begin{definition}\label{def:start and end}
Consider the setup in Proposition \ref{prop:the partition}.
For $P \in \mathcal P$ write $\partial P = \{e^{i\beta},e^{i\gamma}\}$ so that $0 \leq \beta < 2\pi$ and $0 \leq \gamma-\beta < 2\pi$. We denote $St(P) = e^{i\beta}$ the \textdef{start} of $P$ and $En(P) = e^{i\gamma}$ the \textdef{end} of $P$.
We denote $St_\measuredangle(P)=\beta$ and $En_\measuredangle(P)=\gamma$.
\end{definition}

\begin{definition}\label{def:representative points}
Consider the setup in Proposition \ref{prop:the partition}.
\begin{itemize}
\item If $P=P_n^i$ then let $\delta_P=\frac{1}{2}(St_\measuredangle(P)+En_\measuredangle(P))$ and $x_P=e^{i\delta_P}$.
\item If $P=P_n^{|\mathcal P_n|} \cup P_{n+1}^1$ then let $x_P$ be the unique point in $P_n^{|\mathcal P_n|} \cap P_{n+1}^1$.
\end{itemize}
We call each $x_P\in P$ the \textdef{representative point of $P$}.
\end{definition}

\begin{definition}\label{def:refinement}
Consider the setup in Proposition \ref{prop:the partition}.
Let $U$ be the following representation of $\ARt$:
\begin{align*}
U(x) &= V(x_P) \text{ where }x\in P \\ U(x,y) &= V(x_P,x_{P'}) \text{ where }x\in P, y\in P'.
\end{align*}
Note if $x,y \in P$ then $U(x,y)$ is the identity.
We call $U$ the \textdef{refinement} of $V$.
\end{definition}

\begin{proposition}\label{prop:refinement}
Let $\ARt$ be a continuous quiver of type $\At$, $V$ a finitistic representation of $\ARt$, and $U$ the refinement of $V$.
Then $U\cong V$.
\end{proposition}
\begin{proof}
First, let $y\preceq x\in P_n^i$ for some $P_n^i\in \mathcal P_n$.
Then $x,y\in \overline{R}_n$ for some $R_n\in\mathcal R$.
By definition, $x\in\supp A_{n,j}$ if and only if $y\in \supp A_{n,j}$ where the $A_{n,j}$'s are the distinct bar summands of $V|_{\overline{R}_n}$.
Thus, $V(x,y)$ is an isomorphism.
Therefore, if $y\preceq x\in P$ for $P\in\mathcal P_V$ then $V(x,y)$ is an isomorphism.

Let $x\in \Sp^1$ and $P\in\mathcal P_V$ such that $x\in P$.
Define $f(x):U(x_P)=U(x)\to V(x)$ to be $V(x_P,x)$ if $x\preceq x_P$ and to be $V(x,x_P)^{-1}$ if $x_P\preceq x$.
Thus, the following diagram is always commutative:
\begin{displaymath}\xymatrix{
U(x) \ar[d]_-{U(x,y)} \ar@{=}[r] & U(x_P) \ar[d]_-{V(x_P,x_{P'})} \ar[r]^-{f(x)} & V(x) \ar[d]^-{V(x,y)} \\
U(y) \ar@{=}[r] & U(x_{P'}) \ar[r]_-{f(y)} & V(y).
}\end{displaymath}
Therefore, the collection $\{f(x)\}$ yields a morphism $f:U\to V$ and each $f(x)$ is an isomorphism.
By Remark \ref{rmk:isomorphism} we see $f$ is an isomorphism.
\end{proof}

We now briefly recall the definitions of quivers and their representations. Readers are referred to \cite[Chapter II,III]{ASS} for precise details.

\begin{definition}\label{def:quiver stuff}
A \textdef{quiver} $Q=(Q_0,Q_1)$ is a set of vertices $Q_0$ and arrows $Q_1$ (both typically finite) whose source and target are in $Q_0$.
I.e., $Q$ is a directed graph where loops and cycles are allowed. If the underlying undirected graph of $Q$ is a cycle then $Q$ is of type $\At_{|Q_0|-1}$.

A \textdef{representation} $M$ of $Q$ assigns to each vertex $i$ in $Q_0$ a $k$-vector space $M(i)$ and to each arrow $a:i\to j$ in $Q_1$ a $k$-linear transformation $M(a):M(i)\to M(j)$.
\textdef{Morphisms}, \textdef{isomorphisms}, and \textdef{direct sums} of representations of the same quiver are defined similarly to those over $\ARt$.
\end{definition}

\begin{convention}\label{cnv:finite quivers} Let $Q=(Q_0,Q_1)$ be an $\At_n$ quiver.
For convenience, enumerate the vertices counterclockwise $0,1,\ldots,n,n+1$ where $n+1$ is another name for $0\in Q_0$.
If $Q$ has two vertices then it has two arrows.
In this case we always assume that we have fixed one to be between 0 and 1 and the other between 1 and 2, thus moving `around' the underlying cycle.
If $Q$ is an auxiliary quiver (see Definition \ref{def:auxiliary quiver}) we follow the cycle from the continuous type $\At$ quiver.
\end{convention}

\begin{definition}\label{def:auxiliary quiver}
Consider the setup in Proposition \ref{prop:the partition}.
We define an auxiliary quiver $Q_V$. For vertices, $(Q_V)_0 := \mathcal P_V$.
Let $\Delta$ be the diagonal of $(Q_V)_0\times (Q_V)_0$; now we define the arrows of $Q_V$:
\begin{displaymath} (Q_V)_1 = \{(P,P')\in ((Q_V)_0)^2\setminus\Delta : (\exists y\in P', x\in P) \Rightarrow( y\preceq x \text{ and } P\cup P' \text{ is path-connected})\}. \end{displaymath}
Note that $(Q_V)_0$ has at least 2 elements and $Q_V$ is a quiver of type $\At_{|\mathcal P_V|-1}$.
We call $Q_V$ the \textdef{auxiliary quiver} of $V$.
\end{definition}

\begin{definition}\label{def:auxiliary representation}
Consider the setup in Proposition \ref{prop:the partition}.
Let $Q_V$ be the quiver in Definition \ref{def:auxiliary quiver}.
We define a representation $M_V$ of $Q_V$ by defining its vector spaces at each vertex and the linear maps on the arrows of $Q_V$:
\begin{align*}
M_V(P) &= V(x_P) & M_V(P,P') &= V(x_P,x_{P'}).
\end{align*}
Note $M_V$ is finite-dimensional.
We call $M_V$ the \textdef{auxiliary representation} of $V$.
\end{definition}

\begin{example}\label{ex:auxiliary}
    Let $\ARt$ be a continuous quiver of type $\ARt$ so that $S \neq \emptyset$ and let $V$ be a Jordan cell. Then for all $R_n \in \mathcal R$, each bar in the bar code decomposition of $V|_{\overline{R}_n}$ is supported on all of $\overline{R}_n$. This mean the elements of $\mathcal P_V$, and hence the vertices of the auxiliary quiver $Q_V$, are in bijection with the elements of $S$. Denoting vertices by their corresponding element of $S$, there is an arrow $s_i \rightarrow s_{i+1}$ in $Q_V$ if $i$ is odd and there is an arrow $s_{i+1} \rightarrow s_{i}$ in $Q_V$ if $i$ is even. These ``zig-zag'' quivers of type $\At_n$ appear in the work of Burghelea and Dey \cite{BD}.
\end{example}

\begin{remark}\label{rmk:isomorphic lifts}
    Suppose $V$ and $W$ are isomorphic finitistic representations of $\ARt$. Then Remark \ref{rmk:restriction decomposition} implies that $\mathcal P_V = \mathcal P_W$ and thus $Q_V = Q_W$. It follows that the auxilliary representations of $V$ and $W$ are isomorphic; that is, $M_V \cong M_W$.
\end{remark}

We now recall bars and Jordan cells (Definitions \ref{def:discrete bar} and \ref{def:discrete Jordan cell}) for type $\At_n$ quivers.
These are known to representation theorists as strings and bands, respectively.
\begin{definition}\label{def:discrete bar}
Let $Q=(Q_0,Q_1)$ be an $\At_n$ quiver and $\xi_n:\Z\to Q_0$ be given by $i\mapsto r$ where $i\equiv r\mod (n+1)$.
Let $I=\{i,i+1,\ldots,i+\ell\}$ be an interval in $\Z$ and for each $j\in Q_0$ let $E_n(j)=\xi_n^{-1}(j) \cap I$.

For each $j\in Q_0$ let $M_I(j) = k^{|E(j)|}$.
Give each $M_I(j)$ the standard ordered basis $\mathcal B(j)$ and observe the unique order preserving bijection $E_n(j)\to \mathcal B(j)$.
Denote by $\vecb_i$ the image of $i$ in $\mathcal B(j)$.

For each arrow $j\to j'$ in $Q_1$, where $|j-j'|=1$, define $M_I(j,j')$ on each $\vecb_i\in\mathcal B(j)$:
\begin{displaymath}
    M_I(j,j') \left(\vecb_i\right) = \begin{cases}
    \vecb_{i-j+j'} & i-j+j'\in E_n(j') \\
    0 & \text{otherwise}. \end{cases}
\end{displaymath}
A representation isomorphic to $M_I$ is called a \textdef{bar}; a direct sum of bars a \textdef{bar code}.
\end{definition}

\begin{definition}\label{def:discrete Jordan cell}
Let $Q=(Q_0,Q_1)$ be an $\At_n$ quiver.
Let $d\geq 1\in \N$ and let $M$ be a representation of $\At_n$ such that $\dim M(j)=d$ for all $j\in Q_0$. Furthermore, assume $M(j,j')$ is an isomorphism for all arrows $j\to j'\in Q_1$.

By inverting necessary isomorphisms, there is an isomorphism
\begin{displaymath}
\widehat M: M(0)\to M(1)\to \cdots \to M(n)\to M(n+1)=M(0).
\end{displaymath}
We call $M$ a \textdef{Jordan cell} if whenever $M(0)\cong U\oplus W$, where $U$ and $W$ are invariant subspaces under $\widehat{M}$, then either $U=0$ or $W=0$.
\end{definition}

The decomposition of finite-dimensional representations of $\At_n$ type quivers over an algebraically closed field is due to Donovan and Freislich \cite{DF} and Nazarova \cite{N}.
However, algebraic closure is not necessary.
See, for example, \cite{BD} and \cite{G} for acyclic and cyclic cases, respectively.
\begin{theorem}[Combining \cite{DF, N,BD,G}]\label{thm:finite decomposition}
A finite-dimensional representation of an $\At_n$ type quiver decomposes uniquely (up to isomorphism) into a direct sum of a bar code and Jordan cells.
Each bar and each Jordan cell are indecomposable.
\end{theorem}

\section{Main Results}\label{sec:the thms}
In this section we prove the main results via a method called ``pushing down.''
We show how the finite-dimensional indecomposable representations of a type $\At_n$ quiver are related to the  pwf indecomposable 
representations of a type $\ARt$ quiver.

\begin{definition}\label{def:push down}
Let $\ARt$ be a continuous quiver of type $\At$ and $V$ a finitistic representation of $\ARt$.
Let $Q_V$ and $M_V$ be the auxiliary quiver and representation, respectively.
Let $M\cong \bigoplus_{i=1}^n M_i^{\ell_i}$ be a decomposition of $M_V$ as in Theorem \ref{thm:finite decomposition}, where $\ell_i$ denotes multiplicity and each $M_i$ is a bar or Jordan cell.

For each $M_i$ we define a representation $U_i$ of $\ARt$.
We define its vector spaces and maps between vector spaces in adjacent partitions of $\Sp^1$.
Using their composition yields the whole representation $U_i$.
\begin{align*}
U_i(x) &= M_i(P) \text{ where }x\in P \\
U_i(x,y) &= M_i(P,P') \text{ where }x\in P,\, y\in P'.
\end{align*}
We call $U_i$ the \textdef{push down} of $M_i$.
Note if $y\preceq x\in P$ then $U_i(x,y)$ is the identity.
\end{definition}

\begin{lemma}\label{lem:push down}
Consider the set up in Definition \ref{def:push down}.
(1) If $M_i$ is a bar then $U_i$ is a bar.
(2) If $M_i$ is a Jordan cell then $U_i$ is a Jordan cell.
\end{lemma}

\begin{proof}
     (1) Suppose $M_i = M_{\widetilde{I}}$ for some interval $\widetilde{I} \subset \mathbb Z$. By reindexing the vertices of $Q_V$ and rotating the quiver $\ARt$ 
     if necessary,  we can assume $\widetilde{I} = \{0,1,\ldots,\ell\}$ for some $\ell$ and $St_\measuredangle(P_0) = 0$.  Thus there exists $m \geq 0$ such that $m|\mathcal P|\leq \ell < (m+1)|\mathcal P|$.
    
    Now define
        $I := |St_{\measuredangle}(P_0), 2\pi m + En_{\measuredangle}(P_\ell)|$
    where the interval is closed on the left if and only if $St(P_0) \in P_0$ and is closed on the right if and only if $En(P_\ell) \in P_\ell$. We claim that $U_i = M_{I}$.
    
    For $x \in \Sp^1$, write $x = e^{i\theta}$ with $0 \leq \theta < 2\pi$ and suppose $x \in P_j$. Then there is an order preserving bijection $F_x:\xi_n^{-1}(j)\cap \widetilde{I} \rightarrow \xi^{-1}(\theta)\cap I$ given by $F_x(j + \ell|\mathcal P|) = \theta + 2\pi\ell$. Thus we can identify the basis element $\vecb_i \in M_{\widetilde{I}}(j) = U_i(x)$ with $\vecb_{F_x(i)} \in M_{I}(x)$. With this identification, we have $U_i(x) = M_I(x)$.
    
    Now suppose $y \preceq x$ in $\ARt$. Suppose $x \in P_j$ and $y \in P_{j'}$ and write $x = e^{i\theta}, y = e^{i\phi}$ with $\theta, \phi \in [0,2\pi)$. If $j = j'$, then $U_i(x,y)$ is the identity and thus $M_I(x,y)(\vecb_{F_x(i)}) = \vecb_{F_y(i)}$ with $F_x(i) - F_y (i) = \theta - \phi$. Likewise, if $j' = j \pm n$, then $M_I(x,y)(\vecb_{F_x(i)}) = \vecb_{F_y(i \pm n)}$ and $F_x(i) - F_y(i\pm n) = \theta - \phi$. Thus $M_I(x,y) = U_i(x,y)$ as claimed.
    
	(2) Let $M_i$ be a Jordan cell. Assume, without loss of generality, that $e^{i\alpha_0} \in P_0$.
	It follows immediately that the push down $U_i$ has $\widehat{U}_i=\widehat{M}_i$ satisfying Definition \ref{def:band indecomposable} and is thus a Jordan cell.
\end{proof}

\begin{lemma}\label{lem:indecomposable push down}
Each $U_i$ in Definition \ref{def:push down} is indecomposable.
\end{lemma}

\begin{proof}
Suppose $U_i\cong A\oplus B$.
Since $U_i$ is constant on each $P\in \mathcal P_V$ we may assume both $A$ and $B$ are also constant on each $P\in\mathcal P$.
Let $f:U_i\stackrel{\cong}{\to} A\oplus B$ be an isomorphism.
Define $M_A$ and $M_B$ to be representations of $Q_V$ given by
\begin{align*}
M_A(P) &= A(x_P) & M_A(P,P') &= A(x_P, x_{P'}) \\
M_B(P) &= B(x_P) & M_B(P,P') &= B(x_P,x_{P'}).
\end{align*}
By the following commutative diagram we see $M_i\cong M_A\oplus M_B$:
\begin{displaymath}\xymatrix@R=6ex{
M_i(P) \ar@{=}[r] \ar[d]|-{M_i(P,P')} & U_i(x_P) \ar[r]^-{f(x_P)} \ar[d]|-{U(x_P, x_{P'})} &
A(x_P) \oplus B(x_P) \ar@{=}[r] \ar[d]|-{A(x_P,x_{P'})\oplus B(x_P,x_{P'})} & M_A(P)\oplus M_B(P') \ar[d]|-{M_A(P,P')\oplus M_B(P,P')} \\
M_i(P') \ar@{=}[r] & U_i(x_{P'}) \ar[r]_-{f(x_{P'})} & A(x_{P'})\oplus B(x_{P'}) \ar@{=}[r] & M_A(P')\oplus M_B(P').
}\end{displaymath}
Since $M_i$ is indecomposable, up to symmetry $M_i\cong M_A$ and $M_B=0$.
Then $B=0$ and so $U_i\cong A$.
Therefore, $U_i$ is indecomposable.
\end{proof}

\begin{theorem}\label{thm:indecomposables}
Let $V$ be a bar or Jordan cell representation of $\ARt$.
Then $V$ is indecomposable.
\end{theorem}

\begin{proof}
If $\supp V\subset R_n$ for some $R_n\in\mathcal R$ the result follows from Crawley-Boevey's bar code theorem \cite{C-B} and Remark \ref{rmk:restriction decomposition}.
Thus we can assume $V$ is finitistic and, without loss of generality, that $V$ is its own refinement.
Let $\mathcal P_V$, $Q_V$ and $M_V$ be the partition of $\Sp^1$, auxiliary quiver, and auxiliary representation, respectively, obtained from $V$.
Further assume without loss of generality that $e^{i\alpha_0}\in P_0$.

Suppose $V\cong M_I$ is a bar. 
Without loss of generality, assume $\inf(I)\in [0,2\pi)$. 
Then $\inf(I) = St_\measuredangle (P_i)$ for $P_i\in\mathcal P_V$ where $e^{i\cdot \inf(I)}\in P_i$ if and only if $\inf(I)\in I$.
Similarly there exists $m\in\N$ such that $\sup(I)=En_\measuredangle(P_j)+ 2m\pi$ for $P_j\in \mathcal P_V$ where $e^{i\cdot \sup(I)}\in P_j$ if and only if $\sup(I)\in I$.
Let $\widetilde{I} = \{i,i+1,\ldots, j+m|\mathcal P_V|\}$.
Then for each $P\in\mathcal P_V$ and $x\in P$ we see $|E(x)| = |E_n(P)|$.
One may use a similar technique to that in Lemma \ref{lem:push down} to see $M_{\widetilde{I}}$ is a bar.
By Lemma \ref{lem:push down} and its proof the push down of $M_{\widetilde{I}}$ is $M_I$.
By Lemma \ref{lem:indecomposable push down}, $V\cong M_I$ is indecomposable.

Now suppose $V$ is a Jordan cell.
Then $\widehat{M}_V$ in Definition \ref{def:discrete Jordan cell} is equal to $\widehat{V}$.
Thus, $M_V$ is a Jordan cell and its push down is $V$.
By Lemma \ref{lem:indecomposable push down}, $V$ is indeomposable.
\end{proof}
	
\begin{lemma}\label{lem:decomposition push down}
Consider the set up in Definition \ref{def:push down}. 
Let $U$ be the refinement of $V$.
Then $U\cong \bigoplus_{i=1}^n U_i^{\ell_i}$.
\end{lemma}

\begin{proof}
Denote the isomorphism $M\stackrel{\cong}{\to} \bigoplus_{i=1}^n M_i^{\ell_i}$ by $\tilde{f}$.
Let $x\in \Sp^1$ and $P\in\mathcal P_V$ such that $x\in P$.
Define $f(x):U(x)\to \bigoplus U_i^{\ell_i}(x)$ to be $\tilde{f}(P): M(P)\to \bigoplus_{i=1}^n M_i^{\ell_i}(P)$.
Notice $U(x) = M(P)$ and $\bigoplus U_i^{\ell_i}(x) = \bigoplus_{i=1}^n M_i^{\ell_i}(P)$.

\begin{displaymath}\xymatrix{
U(x) \ar@{=}[r] \ar[d]_-{U(x,y)} \ar@/^4ex/[rrr]|-{f(x)}& M_V(P) \ar[r]_-{\tilde{f}(P)} \ar[d]_-{M_V(P,P')} &
\bigoplus M_i^{\ell_i}(P) \ar@{=}[r] \ar[d]^-{\bigoplus M_i(P,P')} & \bigoplus U_i^{\ell_i}(x) \ar[d]^-{\bigoplus U_i(x,y)} \\
U(y) \ar@{=}[r] \ar@/_4ex/[rrr]|-{f(y)} & M_V(P') \ar[r]^-{\tilde{f}(P')} &
\bigoplus M_i^{\ell_i}(P') \ar@{=}[r] & \bigoplus U_i^{\ell_i}(y)
}\end{displaymath}
Then the diagram above commutes, completing the proof.
\end{proof}

\begin{theorem}\label{thm:decomposition}
Let $\ARt$ be a continuous quiver of type $\At$ and $V$ a pointwise finite-dimensional representation of $\ARt$.
Then $V$ decomposes uniquely (up to isomorphism) into a direct sum of a bar code and finitely-many Jordan cells.
\end{theorem}

\begin{proof}
Let the $R_n$'s, the $\overline{R}_n$'s, and $\mathcal R$ be as Definition \ref{def:the R's}.
For each $V_n:= V|_{\overline{R}_n}$ let $W_n$ be the direct sum of all bars in the bar code of $V_n$ whose support is contained in $R_n$.
Then by Lemma \ref{lem:good split} each summand of each $W_n$ is a summand of $V$.
Thus,
\begin{displaymath} V \cong V' \oplus \left( \bigoplus_{R_n\in\mathcal R} W_n \right), \end{displaymath}
where $V'$ is finitistic.
Notice that $\left( \bigoplus_{R_n\in\mathcal R} W_n \right)$ is a bar code but may contain uncountably many bars.

Since $V'$ is finitistic let $U$ be its refinement.
By Lemma \ref{lem:decomposition push down} and Remark \ref{rmk:isomorphic lifts}, there is a unique decomposition (up to isomorphism) $U\cong \bigoplus_{i=1}^m U_i^{\ell_i}$ where each $U_i$ is a bar or Jordan cell.
Since $V'\cong U$ by Proposition \ref{prop:refinement}, we have
\begin{displaymath} V\cong \left( \bigoplus_{i=1}^m U_i^{\ell_i} \right) \oplus \left(\bigoplus_{R_n\in \mathcal R} W_n \right). \end{displaymath}
Since each of the decompositions we have used is unique up to isomorphism so is the whole decomposition.
In the decomposition of $M_V$ there are only finitely-many summands.
Thus, only finitely-many of the $U_i$'s (which each have finite multiplicity) may be Jordan cells.
\end{proof}

\begin{remark}
The pwf requirement is the most general hypothesis one may use.
I.e., without this assumption one must already know something about the decomposition of an $\Sp^1$ persistence module in order to show it decomposes into a bar code and Jordan cells.
See, for example, \cite[Example 4.7]{CC-BdS} or \cite[Proposition 4.2.5]{IRT}.
\end{remark}

\begin{corollary}\label{cor:complete indecomposable classification}
Let $V$ be a pointwise finite-dimensional representation of $\ARt$.
Then $V$ is indecomposable if and only if $V$ is a bar or Jordan cell.
\end{corollary}
\begin{proof}
Combine Theorems \ref{thm:indecomposables} and \ref{thm:decomposition}.
\end{proof}

\begin{corollary}\label{cor:local endomorphism ring}
A pointwise finite-dimensional representation of $\ARt$ has a local endomorphism ring if and only if it is either a bar or a Jordan cell.
\end{corollary}
\begin{proof}
    Botnan and Crawley-Boevey proved that a persistence module over a small category has a local endomorphism ring if and only if it is indecomposable.
    Since $\ARt$ is a small category, Corollary \ref{cor:complete indecomposable classification} implies the result.
\end{proof}

\subsection*{Acknowledgements} The authors would like to thank Francesco Sala and Jun Zhang for pointing out helpful references.

\end{document}